\documentclass{amsart}

\usepackage{amsmath,amssymb,amsthm}
\usepackage[colorlinks,allcolors=blue,bookmarksdepth=2]{hyperref}
\usepackage[capitalise,noabbrev]{cleveref}

\numberwithin{equation}{section}

\theoremstyle{plain}
\newtheorem{theorem}[equation]{Theorem}
\newtheorem{lemma}[equation]{Lemma}
\newtheorem{proposition}[equation]{Proposition}
\newtheorem{corollary}[equation]{Corollary}
\newtheorem{conjecture}[equation]{Conjecture}
\theoremstyle{definition}
\newtheorem{definition}[equation]{Definition}
\theoremstyle{remark}
\newtheorem{remark}[equation]{Remark}

\DeclareMathOperator{\wt}{wt}
\DeclareMathOperator{\supp}{supp}

\newcommand{\F}{\mathbb{F}_2}
\newcommand{\B}{\mathcal{B}}
\newcommand{\ip}[2]{\langle #1,#2\rangle}

\title[A counterexample to Tokareva's conjecture]{A counterexample to Tokareva's conjecture on sums of two bent functions}
\author{Yansong Feng}
\address{MiniMax}
\email{ivan@minimax.io}

\author{Hao Jiang}
\address{MiniMax}
\email{dee@minimax.io}

\author{Guancheng Li}
\address{MiniMax}
\email{atum@minimax.io}

\author{Tianyi Li}
\address{MiniMax}
\email{n0b0dy@minimax.io}

\author{Lili Tang}
\address{MiniMax}
\email{tanglee@minimax.io}

\subjclass[2020]{Primary 06E30, 94D10; Secondary 94A60}
\keywords{Boolean function, bent function, Walsh transform, bent-sum decomposition}

\begin{document}

\begin{abstract}
It is well known that the sum of two bent functions on $\F^n$ has algebraic degree at most $n/2$ for even $n\ge 4$. Tokareva conjectured that the converse holds: every Boolean function with degree at most $n/2$ can be written as a sum of two bent functions. The conjecture has been verified for even $n\leq 6$ and for several special classes of Boolean functions. We disprove this conjecture by showing that $x_1x_2(x_3x_4+x_5x_6+x_7x_8)$ is not a sum of two bent functions on $\F^8$. We extend this construction to obtain counterexamples in every even dimension $n\ge 8$. Using Shaporenko's reduction, we also obtain explicit counterexamples to Tokareva's derivative conjecture in every even dimension $n\ge 10$.
\end{abstract}
\maketitle

\section{Introduction}\label{sec:introduction}
Bent functions were introduced by Rothaus~\cite{rothaus1976}, who also proved that for even $n\ge 4$, every bent function on $\F^n$ has algebraic degree at most $n/2$, and hence so does any sum of two bent functions.
Tokareva~\cite{tokareva2011} asked whether the converse holds, which is stated as follows:
\begin{conjecture}[Tokareva~\cite{tokareva2014}]
\label{conj:sum}
Let $n$ be even. Every Boolean function on $\F^n$ of algebraic degree
at most $n/2$ is a sum of two bent functions on $\F^n$.
\end{conjecture}

This question was subsequently recorded as an important open problem by Carlet~\cite[Open Problem~2]{carlet2014open} and revisited by Tokareva in~\cite{tokareva2014,tokareva2015book}.
An affirmative answer would imply that the logarithm of the number of bent functions has the same asymptotic order as that of the size of the Reed--Muller code of degree at most $n/2$. This would close the gap in asymptotic order between the best known lower and upper bounds on this logarithm, a gap that remains substantial despite decades of research~\cite{tokareva2015book}.

Tokareva~\cite{tokareva2011} initially verified \Cref{conj:sum} for $n \le 6$ via computer search. Subsequently, Qu, Fu, Dai, and Li~\cite{qu2016} provided a theoretical proof of the conjecture for all $n \le 6$ as well as for all quadratic functions. For $n = 8$, Tokareva~\cite{tokareva2014} proved that every cubic function can be expressed as the sum of at most four bent functions, while Shaporenko~\cite{shaporenko2022} showed that any bent function of degree at most $3$ can be decomposed into the sum of two bent functions.

In this paper, we show that this conjecture does not hold for $n\geq 8$ by constructing counterexamples. We prove a necessary condition for a Boolean function can be decomposed as the sum of two bent functions. We then construct a function of degree $4$ on $\F^8$,
\[
 f_8(x_1,\ldots,x_8)
  =x_1x_2(x_3x_4+x_5x_6+x_7x_8),
\]
and show that it does not satisfy this condition. Thus, $f_8$ is a counterexample to this conjecture. We also extend this construction to all even $n\geq 8$. Furthermore, there is another conjecture on bent functions proposed by Tokareva~\cite{tokareva2016}, which will be introduced as \Cref{conj:derivative} later. It was related to the above bent-sum decompositions by Shaporenko in~\cite{shaporenko2023}. This conjecture was disproved by Potapov for all sufficiently large $n$ by a counting argument~\cite{potapov2023}. Instead, we construct
an explicit counterexample for $n\geq 10$.

The remainder of this paper is organized as follows. \Cref{sec:pre} introduces the basic notation and definitions for bent functions. In \Cref{sec:congruence}, we prove a necessary condition on the weight of sums of two bent functions. \Cref{subsec:bent-sums} presents a counterexample to \Cref{conj:sum} in every even dimension $n\ge 8$. Finally, \Cref{subsec:derivatives} adapts this construction to provide an explicit counterexample to \Cref{conj:derivative} for all even $n\ge 10$.

\section{Preliminaries}
\label{sec:pre}
Throughout this paper, all vector spaces and polynomials are defined over the binary field $\F$. Let $e_1,\ldots,e_n$ denote the standard basis of $\F^n$, where $e_i$
has $1$ in the $i$-th coordinate and $0$ elsewhere. Every Boolean function $f$ on $\F^n$ has
a unique \emph{algebraic normal form} (ANF), which is represented as a multivariate polynomial
\[
 f(x)=\sum_{I\subseteq\{1,\ldots,n\}}a_I\prod_{i\in I}x_i,
 \qquad a_I\in\F.
\]
The \emph{algebraic degree} of $f$, denoted by $\deg(f)$, is the maximum cardinality $|I|$ for which $a_I\ne 0$. We set $\deg(0)=-\infty$. Furthermore, the \emph{support} of $f$ and its \emph{Hamming weight} are defined, respectively, by
\[
 \supp(f)=\{x\in\F^n:f(x)=1\} \qquad \text{and} \qquad \wt(f)=|\supp(f)|.
\]
The function $f$ is called \emph{balanced} if $\wt(f)=2^{n-1}$.

An \emph{affine} subspace of $\F^n$ is a set $x_0+V$, where $V$ is a linear
subspace of $\F^n$ and $x_0\in\F^n$. Its dimension is $\dim(V)$.

For any two vectors $x,\omega\in\F^n$, we denote their standard inner product by $\ip{x}{\omega}=\sum_{i=1}^n x_i\omega_i$. The
Fourier transform of a real-valued function $\phi$ on $\F^n$ is given by
\[
 \widehat\phi(\omega)=\sum_{x\in\F^n}\phi(x)(-1)^{\ip{x}{\omega}},
 \qquad \omega\in\F^n.
\]

Accordingly, the \emph{Walsh--Hadamard transform} of a Boolean function $f$ is the Fourier transform of its sign function $(-1)^f$, i.e., $W_f=\widehat{(-1)^f}$, or explicitly,
\[
 W_f(\omega)=\sum_{x\in\F^n}(-1)^{f(x)+\ip{x}{\omega}}, \qquad \omega\in\F^n.
\]

The Fourier transform satisfies the following fundamental properties, which are known as the \emph{Fourier inversion formula} and \emph{Parseval's identity}, respectively.

\begin{proposition}
Let $\phi$ be a real-valued function on $\F^n$. Then for all $x\in\F^n$,
\[
 \phi(x)=2^{-n}\sum_{\omega\in\F^n}\widehat\phi(\omega)(-1)^{\ip{x}{\omega}}.
\]
Moreover,
\[
 \sum_{\omega\in\F^n}\widehat\phi(\omega)^2 = 2^n\sum_{x\in\F^n}\phi(x)^2.
\]
\end{proposition}

\begin{definition}
Let $n$ be even. A Boolean function $f$ on $\F^n$ is \emph{bent} if
$|W_f(\omega)|=2^{n/2}$ for every $\omega\in\F^n$.
\end{definition}
 We denote the set of all $n$-variable bent functions by $\B_n$. In~\cite{rothaus1976}, Rothaus proved that for even $n\ge 4$, every bent function on $\F^n$ has algebraic degree at most $n/2$, and hence so does any sum of two bent functions.
\begin{proposition}[Rothaus~\cite{rothaus1976}]
\label{prop:bent-degree-bound}
Let $n\ge 4$ be even. Every bent function $f$ on $\F^n$
has algebraic degree at most $n/2$. Consequently, for any
$f_1,f_2\in\B_n$,
\[
 \deg(f_1+f_2)\le \frac{n}{2}.
\]
\end{proposition}

For $a\in\F^n$, the \emph{derivative} of $f$ in the direction $a$ is
\[
 D_a f(x)=f(x)+f(x+a).
\]
If $f$ is nonconstant, then $\deg(D_a f)\le\deg(f)-1$.
Moreover, $D_a f(x+a)=D_a f(x)$ for all $x,a\in\F^n$.

A vector $a$ is a \emph{period} of $f$ if $D_a f=0$, or equivalently,
if $f(x+a)=f(x)$ for all $x\in\F^n$.
The periods of $f$ form a linear subspace of $\F^n$.

For even $n$, a Boolean function is bent if and only if all its
derivatives in nonzero directions are balanced. In particular,
a bent function has no nonzero periods.

\section{A necessary condition for bent-sum decompositions}
\label{sec:congruence}
We will show that if $f$ is the sum of two bent functions on $\F^n$ whose support lies in an affine subspace of dimension $n - k$ with $k<n/2$, then $\wt(f)$ is divisible by $2^{k+1}$. To this end, we first specialize the signed disagreement formula of Kolomeec and Pavlov~\cite{kolomeecpavlov2009} to a pair of bent functions.

\begin{lemma}
\label{lem:sign-restriction}
Let $f = f_1 + f_2$ with $f_1, f_2 \in \B_n$. Define $\phi = \tfrac{1}{2}\bigl((-1)^{f_1} - (-1)^{f_2}\bigr)$. Then $\widehat{\phi}(\omega) \in \{0, \pm 2^{n/2}\}$ for every $\omega \in \F^n$.
\end{lemma}

\begin{proof}
Since $f = f_1 + f_2$, we have $f_2(x) = f_1(x)$ for $x \notin \supp(f)$ and $f_2(x) = f_1(x) + 1$ for $x \in \supp(f)$. Consequently, $(-1)^{f_2(x)}$ equals $(-1)^{f_1(x)}$ for $x \notin \supp(f)$ and $-(-1)^{f_1(x)}$ for $x \in \supp(f)$, which gives
\[
  \phi(x) =
  \begin{cases}
    (-1)^{f_1(x)}, & x \in \supp(f), \\
    0, & x \notin \supp(f).
  \end{cases}
\]

By the linearity of the Fourier transform, we have
\[
  \widehat{\phi}(\omega) = \tfrac{1}{2}\bigl(W_{f_1}(\omega) - W_{f_2}(\omega)\bigr).
\]
Since $f_1, f_2 \in \B_n$, both $W_{f_1}(\omega)$ and $W_{f_2}(\omega)$ belong to $\{\pm 2^{n/2}\}$. Their difference is therefore $0$ or $\pm 2^{n/2+1}$, and dividing by $2$ yields $\widehat{\phi}(\omega) \in \{0, \pm 2^{n/2}\}$.
\end{proof}

Now we prove a necessary condition for a sum of two bent functions, as in the following \Cref{thm:divisibility}.

\begin{theorem}
\label{thm:divisibility}
Let $f = f_1 + f_2$ with $f_1, f_2 \in \B_n$. If $\supp(f)$ is contained in an affine subspace of dimension $n - k$ with $k < n/2$, then $2^{k+1} \mid \wt(f)$.
\end{theorem}

\begin{proof}
Since bent functions and Hamming weight are preserved under invertible affine transformations, we may assume without loss of generality that $\supp(f) \subseteq \{0\}^k \times \F^{n-k}$.

Let $\phi$ be as in \Cref{lem:sign-restriction}, and define $\psi\colon \F^{n-k} \to \{0, \pm 1\}$ by $\psi(z) = \phi(0, z)$. Clearly, $\sum_{z \in \F^{n-k}} \psi(z)^2 = \wt(f)$. Since $\phi$ vanishes outside $\{0\}^k \times \F^{n-k}$, we have $\widehat{\psi}(\beta) = \widehat{\phi}(0, \beta)$ for all $\beta \in \F^{n-k}$, and thus $\widehat{\psi}(\beta) \in \{0, \pm 2^{n/2}\}$ by \Cref{lem:sign-restriction}.

By Parseval's identity on $\F^{n-k}$, we have
\[
  \sum_{\beta \in \F^{n-k}} \widehat{\psi}(\beta)^2 = 2^{n-k} \sum_{z \in \F^{n-k}} \psi(z)^2 = 2^{n-k}\wt(f).
\]
Dividing both sides by $2^n$ and noting that $2^{-n/2}\widehat{\psi}(\beta) \in \{0, \pm 1\}$, we obtain
\[
  2^{-k}\wt(f) = \sum_{\beta \in \F^{n-k}} \bigl(2^{-n/2}\widehat{\psi}(\beta)\bigr)^2
  \equiv \sum_{\beta \in \F^{n-k}} 2^{-n/2}\widehat{\psi}(\beta) \pmod 2.
\]
Moreover, by the Fourier inversion formula at $0$, we have
\[
  \sum_{\beta \in \F^{n-k}} 2^{-n/2}\widehat{\psi}(\beta) = 2^{n/2-k}\psi(0).
\]
Since $k < n/2$ and $\psi(0) \in \{0, \pm 1\}$, the right-hand side is an even integer, and hence $2^{k+1} \mid \wt(f)$.
\end{proof}

\begin{remark}
When $k=0$, \Cref{thm:divisibility} simply recovers the well-known fact that the sum of any two bent functions has even weight.
\end{remark}

We conclude this section with a special case of \Cref{thm:divisibility}, which will be used to construct our counterexamples in the next section.

\begin{corollary}
\label{cor:cylinder}
Let $n$ be even and $1 \le k < n/2$, and let $s$ be a Boolean function on $\F^{n-k}$. If the function $f$ on $\F^n$ defined by
\[
  f(x_1,\ldots,x_n) = x_1\cdots x_k\,s(x_{k+1},\ldots,x_n)
\]
is a sum of two bent functions, then $2^{k+1} \mid \wt(s)$.
\end{corollary}

\begin{proof}
Note that
\[
    \supp(f)
    = \bigl\{(\underbrace{1,\ldots,1}_{k\text{ times}})\bigr\}
      \times \supp(s).
\]
Thus $\supp(f)$ is contained in an affine subspace of dimension
$n-k$, and $\wt(f)=\wt(s)$. The result follows immediately from
\Cref{thm:divisibility}.
\end{proof}

\section{Counterexamples}
\label{sec:counterexamples}

\subsection{Sums of two bent functions}
\label{subsec:bent-sums}
For $\ell\ge1$, let $q_\ell$ denote the standard quadratic form on $\F^{2\ell}$ given by the direct sum of $\ell$ disjoint quadratic monomials,
\[
 q_\ell(z_1,\ldots,z_{2\ell})
   =z_1z_2+z_3z_4+\cdots+z_{2\ell-1}z_{2\ell}.
\]

\begin{lemma}
\label{lem:quadratic-weight}
For every $\ell\ge1$, the function $q_\ell$ is bent on $\F^{2\ell}$, and its weight is $\wt(q_\ell)=2^{2\ell-1}-2^{\ell-1}$.
\end{lemma}

\begin{proof}
For any $\omega\in\F^{2\ell}$, the Walsh transform of $q_\ell$ factors over the $\ell$ disjoint pairs of variables as
\[
 W_{q_\ell}(\omega)
 =\prod_{j=1}^\ell\Bigl(\sum_{a,b\in\F}
   (-1)^{ab+a\omega_{2j-1}+b\omega_{2j}}\Bigr)
 =(-1)^{q_\ell(\omega)}\,2^\ell.
\]
In particular, $|W_{q_\ell}(\omega)|=2^\ell$ for all $\omega\in\F^{2\ell}$, so $q_\ell$ is bent on $\F^{2\ell}$.
Taking $\omega=0$ and using
\[
 W_{q_\ell}(0)=2^{2\ell}-2\wt(q_\ell)
\]
gives the proof.
\end{proof}

\begin{theorem}
\label{thm:all-even}
Let $n\ge 8$ be even and write $n=2m$. The Boolean function on $\F^n$ defined by
\begin{equation}
 \label{eq:top-degree}
 f_n(x)
  =x_1x_2\cdots x_{m-2}\;q_3(x_{m-1},x_m,\ldots,x_{m+4})
\end{equation}
has algebraic degree $m$ and cannot be written as a sum of two bent functions on $\F^n$.
\end{theorem}

\begin{proof}
To apply \Cref{cor:cylinder} with $k=m-2$ (note that $1\le k<m$ and $n-k=m+2$), take the Boolean function $s$ on $\F^{m+2}$ defined by
\[
 s(x_{m-1},\ldots,x_n)=q_3(x_{m-1},\ldots,x_{m+4}).
\]
For each fixed assignment to $x_{m-1},\ldots,x_{m+4}$, the remaining $m-4$ variables can be chosen arbitrarily without changing the value of $s$. Thus each assignment for which $q_3=1$ gives $2^{m-4}$ assignments for which $s=1$. By \Cref{lem:quadratic-weight} with $\ell=3$, we have $\wt(q_3)=2^5-2^2=28$, and hence
\[
 \wt(s)=2^{m-4}\wt(q_3)=28\cdot 2^{m-4}=7\cdot 2^{m-2}.
\]
Thus $2^{m-1}\nmid\wt(s)$, and the result follows from \Cref{cor:cylinder}.
\end{proof}

Thus \Cref{conj:sum} fails for every even $n\ge 8$. Since the conjecture is known to hold for $n\le 6$, $n=8$ is the smallest dimension in which it fails.

\begin{remark}
For $n=8$, the polynomial $f_8=x_1x_2(x_3x_4+x_5x_6+x_7x_8)$ is already known in the literature. In~\cite{carlet2025}, Carlet, Feukoua, and S\u{a}l\u{a}gean studied it as a quartic whose degree drops on the hyperplanes $x_1=0$, $x_2=0$, and $x_1+x_2=0$. Up to a permutation of variables, it also appears in the classification of eight-variable quartic forms by Langevin, Leander, Rabizzoni, V\'eron, and Zanotti~\cite{langevincatalogue}. Our new contribution here is showing that $f_8$ cannot be written as a sum of two bent functions.
\end{remark}

\subsection{Derivatives of bent functions}
\label{subsec:derivatives}

Every derivative of a bent function on $\F^n$
in a nonzero direction is balanced, has algebraic degree at most
$n/2-1$, and has a nonzero period. Tokareva~\cite{tokareva2016}
conjectured that these necessary conditions are sufficient.

\begin{conjecture}[Tokareva~\cite{tokareva2016}]
\label{conj:derivative}
Let $n$ be even, and let $f$ be a balanced Boolean function on $\F^n$
of degree at most $n/2-1$. If $f$ has a nonzero period, then $f$ is
a derivative of a bent function on $\F^n$.
\end{conjecture}

Potapov~\cite{potapov2023} disproved \Cref{conj:derivative} for all
sufficiently large $n$ by a counting argument. This argument establishes
the existence of counterexamples, but does not give an explicit one.

Shaporenko~\cite{shaporenko2023} proposed a connection between
derivatives of bent functions and sums of two bent functions. Here we use the following special case.

Write $x=(x_1,x_2,z)\in\F\times\F\times\F^{n-2}$ with $z=(x_3,\ldots,x_n)\in\F^{n-2}$ and let $e_1=(1,0,\ldots,0)$.

\begin{lemma}[Shaporenko~\cite{shaporenko2023}]
\label{lem:shaporenko}
Let $f$ be a Boolean function on $\F^{n-2}$, and let $h(x_1,x_2,z)=x_2+f(z)$. Then $h=D_{e_1}g$ for some bent
function $g$ on $\F^n$ if and only if $f\in\B_{n-2}+\B_{n-2}$.
\end{lemma}

\begin{proof}
Suppose $h=D_{e_1}g$ with $g$ bent, and put $v(x_2,z)=g(0,x_2,z)$.
Then $g(1,x_2,z)=v(x_2,z)+x_2+f(z)$, so
\[
 (-1)^{g(x_1,x_2,z)}=(-1)^{v(x_2,z)+x_1(x_2+f(z))}.
\]
Let $\omega=(c,d,\beta)\in\F\times\F\times\F^{n-2}$. The sum over
$x_1$ vanishes unless $x_2+f(z)=c$, so
\[
 W_g(\omega)
 =2\sum_{z\in\F^{n-2}}
   (-1)^{v(f(z)+c,\,z)+(f(z)+c)d+\ip{z}{\beta}}
 =\pm 2\,W_{v_c+df}(\beta),
\]
where $v_c(z)=v(f(z)+c,z)$ and the sign is $(-1)^{cd}$.
Because $|W_g(\omega)|=2^{n/2}$ for every $\omega$, the functions $v_0$
and $v_0+f$ have all Walsh coefficients of absolute value
$2^{(n-2)/2}$. Both are therefore bent on $\F^{n-2}$, and
$f=v_0+(v_0+f)\in\B_{n-2}+\B_{n-2}$.

Conversely, let $f=b_1+b_2$ with $b_1,b_2\in\B_{n-2}$, and put
$g(x_1,x_2,z)=b_1(z)+x_1\bigl(x_2+f(z)\bigr)$. Then
$D_{e_1}g=x_2+f=h$. The computation above, with
$v(x_2,z)=b_1(z)$, gives $v_0=v_1=b_1$ and
$|W_g(\omega)|=2\,|W_{b_1+df}(\beta)|=2^{n/2}$, since $b_1$ and
$b_1+f=b_2$ are bent. So $g$ is bent.
\end{proof}

Combining \Cref{lem:shaporenko} with \Cref{thm:all-even} gives explicit counterexamples to \Cref{conj:derivative}.

\begin{theorem}
\label{thm:derivative}
Let $n\ge 10$ be even and write $n=2m$. The Boolean function on $\F^n$ defined by
\begin{equation}
 \label{eq:hn}
 h_n(x)
  =x_2+f_{n-2}(x_3,\ldots,x_n)
  =x_2+x_3\cdots x_{m-1}\;q_3(x_m,\ldots,x_{m+5})
\end{equation}
is balanced, has algebraic degree $m-1$ and a nonzero period,
and cannot be written as a derivative of a bent function on $\F^n$.
\end{theorem}

\begin{proof}
For fixed $x_1,x_3,\ldots,x_n$, each value of $x_2$ gives $h_n(x)=1$, which implies $h_n$ is balanced.
Expanding \eqref{eq:hn}, we see that $h_n$ is the sum of
the linear term $x_2$ and three distinct monomials of degree $(m-3)+2=m-1$. Hence $\deg(h_n)=m-1$.

We first show that the period space of $h_n$ is
\begin{equation}
 \label{eq:periods}
 K=\{a\in\F^n:a_2=\cdots=a_{m+5}=0\}.
\end{equation}
Every vector in $K$ is clearly a period, since $h_n$ does not
depend on the corresponding coordinates. Conversely, let
$a=(a_1,a_2,b)$ be a period of $h_n$, with
$b=(b_3,\ldots,b_n)$. Then
\[
 D_bf_{n-2}=a_2.
\]
By the weight computation in the proof of \Cref{thm:all-even},
\[
 \wt(f_{n-2})=7\cdot 2^{m-3}\ne 2^{n-3}.
\]
Thus $f_{n-2}$ is not balanced, which rules out
$D_bf_{n-2}=1$. Hence $a_2=0$, so $b$ is a period of $f_{n-2}$.
Since $f_{n-2}$ is nonzero and vanishes unless
$x_3=\cdots=x_{m-1}=1$, we must have
$b_3=\cdots=b_{m-1}=0$.
Setting $x_3=\cdots=x_{m-1}=1$ in $D_bf_{n-2}=0$ then gives
\[
 D_{(b_m,\ldots,b_{m+5})}q_3=0.
\]
By \Cref{lem:quadratic-weight}, $q_3$ is bent and hence has
no nonzero period. Thus $b_m=\cdots=b_{m+5}=0$,
which proves \eqref{eq:periods}.

Now suppose that $h_n=D_ag$ for some bent function $g$ and
some $a\ne0$. Then $D_ah_n=0$, so $a\in K$.
Choose an invertible linear map on $K$ taking $e_1$ to $a$,
and extend it by the identity on
$\langle e_2,\ldots,e_{m+5}\rangle$ to an invertible linear
map $T$ on $\F^n$.
Since $h_n$ depends only on $x_2,\ldots,x_{m+5}$, we have
$h_n\circ T=h_n$, and hence
\[
 D_{e_1}(g\circ T)=(D_ag)\circ T=h_n.
\]
As $g\circ T$ is bent, \Cref{lem:shaporenko} implies
$f_{n-2}\in\B_{n-2}+\B_{n-2}$, contradicting
\Cref{thm:all-even}.
\end{proof}
Thus \Cref{conj:derivative} fails for every even $n\ge 10$.

\subsubsection*{Use of AI}
The counterexamples and proofs in this paper were obtained through large-scale automated reasoning with LLM agents while solving Problem~10 in NSUCRYPTO~2026 (MiniMax M3.1-Flash-Preview, 19 rounds of up to 100 independent agents between 12 and 18 October 2026, 1889 reports in all), paired with 8 rounds of independent audits (227 reports).
In the first 10 rounds, the agents decomposed every cubic on $\F^8$ and our first quartic counterexample candidate into two bent functions. When $f_8$ was later examined, an early attempt to prove that $f_8\notin\B_8+\B_8$ was discarded during the audit phase because it mistakenly treated the Walsh transform as additive. Another agent then established \Cref{lem:sign-restriction} and confirmed that $f_8$ is indeed a counterexample, which in turn led to the divisibility condition in \Cref{thm:divisibility}, the general family in \Cref{thm:all-even}, and the derivative counterexamples in \Cref{thm:derivative}.
Every statement in this paper was subsequently checked by the authors. We also provide a formal verification for \Cref{thm:all-even} and \Cref{thm:derivative} in Lean~4. The authors take full responsibility for the contents.
\bibliographystyle{amsplain}
\bibliography{refs}
\end{document}